\documentclass{amsart}
\usepackage{amsmath,amsthm}
\usepackage{amsfonts,amssymb}
\usepackage{enumerate}
\usepackage{makecell}

\usepackage{cite}

\usepackage{array}
\newcolumntype{M}{>{\centering\arraybackslash}m{\dimexpr.25\linewidth-2\tabcolsep}}
  
\newtheorem{theorem}{Theorem}[section]

\newtheorem{lemma}[theorem]{Lemma}
\newtheorem{proposition}[theorem]{Proposition}

\newtheorem{remark}[theorem]{Remark}

\theoremstyle{remark}

\newcommand{\mR}{\mathbb{R}}

\newcommand{\mE}{\mathbb{E}}

\newcommand{\mS}{\mathbb{S}}

\newcommand{\cP}{\mathcal{P}}

\newcommand{\ux}{\underline{x}}

\newcommand{\uo}{\underline{\omega}}
\newcommand{\uD}{\underline{D}}

\newcommand{\bfx}{\mathbf{x}}

\newcommand{\upx}{\partial_{\underline{x}}}

\newcommand{\pI}{\partial_{x_i}}
\newcommand{\pR}{\partial_{r}}

\begin{document}
 
\title
{Algebraic approach to slice monogenic functions}
\author{Lander Cnudde *}
\email{Lander.Cnudde@UGent.be}
\author{Hendrik De Bie *}
\email{Hendrik.DeBie@UGent.be}
\author{Guangbin Ren $^\ddagger$}
\email{rengb@ustc.edu.cn}
\address{* Department of Mathematical Analysis\\Faculty of Engineering and Architecture\\Ghent University\\Galglaan 2, 9000 Gent\\ Belgium.}
\address{$^\ddagger$ Department of Mathematics\\ USTC, 96 Jinzhai Road, Hefei, Anhui 230026 \\ P.R.China}

\date{\today}
\keywords{Slice Dirac operator, Clifford-Hermite polynomials, slice monogenic functions}

\begin{abstract}
In recent years, the study of slice monogenic functions has attracted more and more attention in the literature. In this paper, an extension of the well-known Dirac operator is defined which allows to establish the Lie superalgebra structure behind the theory of slice monogenic functions. Subsequently, an inner product is defined corresponding to this slice Dirac operator and its polynomial null-solutions are determined. Finally, analogues of the Hermite polynomials and Hermite functions are constructed in this context and their properties are studied.

\end{abstract}

\maketitle

\section{Introduction}
\setcounter{equation}{0}

Recently, a lot of attention has been paid to developing a new theory of slice monogenic functions (see e.g.\cite{colombo2009slice} and the book \cite{colombo2011noncommutative}). It entails a generalisation of the theory of complex analysis to higher dimensions, where Clifford algebras take over the role played by the complex numbers while still preserving the essential features of complex analysis. Up to now, the focus in this line of research was in establishing analytic results in the flavour of complex analysis (see \cite{rungetheorem,gentili,quaternionic}) or distribution theory (see \cite{distr}), as well as the development of a new functional calculus for noncommutative operators (see \cite{colombo2009slice, colombo2011noncommutative, spectrum, newfunctional, smonogenickernel}). Also more elaborate generalisations, e.g. using real alternative algebras, have been investigated (see \cite{alternative}).
In other approaches to hypercomplex analysis, an important role is played by an underlying algebraic structure, namely the Lie superalgebra $\mathfrak{osp}(1\lvert 2)$, which allows to find a representation theoretic interpretation of various function space decompositions (see \cite{3032075}). It also allows for the introduction of a generalized Fourier transform (see e.g. the review \cite{de2012clifford}).
Therefore, the main aim of the present paper is to show how the same algebraic structure appears in the theory of slice monogenic functions. This is achieved by rewriting the relevant differential operator (see \cite{colombo2013nonconstant}), whose kernel consists of the set of slice monogenic functions, in a more suitable form. We also construct the main ingredients to establish a Fourier transform in a subsequent paper, such as a proper Hilbert module of functions and a subspace spanned by the set of so-called Clifford-Hermite functions.
The theory of slice monogenic functions acts within the framework of an $m$-dimensional real Clifford algebra $Cl_{m}$. This algebra has $m$ basis vectors $\widetilde{e_i},i=1\ldots m$, which satisfy the relations
\begin{equation}\label{commutation}
\widetilde{e_i} \widetilde{e_j} + \widetilde{e_j} \widetilde{e_i} = -2 \delta_{ij}, \quad i,j=1,\ldots, m.
\end{equation}
The notion of a paravector  in $Cl_m$ allows for a higher-dimensional generalisation of the classical Cauchy-Riemann operator $\partial_{\overline{z}}=\frac{1}{2}\left(\partial_x+i\partial_y\right)$ from complex analysis. A paravector $\widetilde{\text{x}}$ is defined as the sum of a scalar $x_0$ and a $1$-vector $\widetilde{\ux}$:
\begin{align*}
\widetilde{\text{x}} &= x_0 + x_1\widetilde{e_1} + \ldots + x_m\widetilde{e_m}\\
&=x_0 + \widetilde{\ux}.
\end{align*}
In this way an element $x=(x_0,x_1,\ldots,x_m)\in \mathbb{R}^{m+1}$ is identified with a paravector $\widetilde{\text{x}} \in Cl_m^{(0)} \oplus Cl_m^{(1)}$, where $Cl_m^{(k)}$ denotes the space of $k$-vectors in $Cl_m$. The scalar part $x_0$ of x can be regarded as the analogue of the real part of a complex number.
A first generalisation of $\partial_{\overline{z}}$ is thus defined as
\begin{equation*}
\widetilde{D}^{CR}=\partial_{x_0} + \sum_{i=1}^m \widetilde{e_i} \pI = \partial_{x_0} + \partial_{\widetilde{\ux}},
\end{equation*}
with $\partial_{\widetilde{\ux}}=\sum_{i=1}^m\widetilde{e_i} \pI$ the classical Dirac operator. This operator is called the \emph{Cauchy-Riemann operator} and its null-solutions are said to be \emph{monogenic} (see e.g. \cite{brackx1982clifford,delanghe1992clifford} for a detailed study of such functions). Because of \eqref{commutation}, $\widetilde{D}^{CR}$ equals the classical Cauchy-Riemann operator $\partial_{\overline{z}}$ when $m=1$.
A second generalisation of $\partial_{\overline{z}}$ is based on the polar form of the paravector. Writing $\widetilde{\ux}=r\widetilde{\uo}$ with $r=|\widetilde{\ux}|=\sqrt{x_1^2+\ldots+x_m^2}$ and $\widetilde{\uo}=\widetilde{\ux}/r$, the scalar $x_0$ is the real part of $\widetilde{\text{x}}=x_0+r\widetilde{\uo}$ and the $1$-vector $\widetilde{\uo}$ behaves as the classical imaginary unit because $\widetilde{\uo}^2=-1$. Therefore we define
\begin{equation*}
\widetilde{D}_0^{CR}=\partial_{x_0} + \widetilde{\uo} \partial_{r}
\end{equation*}
and call this operator the \emph{slice Cauchy-Riemann operator}. Defining the Euler operator in $\mR^{m+1}$ as
\begin{equation*}
\mE=x_0 \partial_{x_0}+\mE_m=x_0 \partial_{x_0}+\sum_{i=1}^m x_i \pI,
\end{equation*}
$\widetilde{D}_0^{CR}$ can be rewritten as
\begin{equation*}
\widetilde{D}_0^{CR}=\partial_{x_0} + \frac{\widetilde{\ux}}{\vert \widetilde{\ux} \vert ^2}\mE_m.
\end{equation*}
\begin{remark}
Note that in \cite{colombo2013nonconstant} the operator $|\widetilde{\ux}|^2\widetilde{D}^{CR}_0$ was introduced. Its null-solutions correspond to the class of slice monogenic functions as studied in \cite{colombo2009slice,colombo2011noncommutative}. The additional factor $|\ux|^2$ does not affect the kernel of $\widetilde{D}^{CR}_0$ and hence null-solutions of $\widetilde{D}^{CR}_0$ are also slice monogenic.
\end{remark}
The main idea in our paper is to replace the paravector $\widetilde{\text{x}} \in Cl_m$ by a $1$-vector $\bfx$ in the higher-dimensional $Cl_{m+1}$-algebra. The element $x=(x_0,x_1,\ldots,x_m)\in \mathbb{R}^{m+1}$ is now identified with the $1$-vector $\bfx \in Cl_{m+1}$ defined as
\begin{align*}
\bfx &= x_0e_0 + x_1e_1 + \ldots + x_me_m\\
&= x_0 e_0 + \ux,
\end{align*}
where $e_i,i=0,\ldots,m$ are the basis vectors of $Cl_{m+1}$ satisfying the relations
\begin{equation*}
e_i e_j + e_j e_i = -2 \delta_{ij}, \quad i,j=0,\ldots, m.
\end{equation*}
One thus has $\bfx^2 = -\vert x \vert^2 = -(x_0^2+r^2)$. 
The procedure to go over from $\widetilde{\text{x}} \in Cl_m$ to $\bfx \in Cl_{m+1}$ consists of left-multiplying $\widetilde{\text{x}}$ with the new basis vector $\widetilde{e_0}$. Because the bivectors $e_i=\widetilde{e_0}\widetilde{e_i}, i=1,\ldots,m$, obey the same relations as the basis vectors $\widetilde{e_i}$ themselves, they can be regarded as the other $m$ basis vectors of $Cl_{m+1}$. 
The transformation 
\begin{align}
\widetilde{e_0}= e_0 \nonumber \\
\widetilde{e_0}\widetilde{e_i}= e_i
\label{transfo}
\end{align}
thus results in the basis $e_i,i=0,\ldots,m$ for $Cl_{m+1}$ and permits a definition of the previous differential operators. For the Cauchy-Riemann operator $\widetilde{D}^{CR}$ the procedure yields
\begin{equation*}
D=e_0\partial_{x_0} + \sum_{i=1}^m e_i \pI =  \sum_{i=0}^m e_i \pI,
\end{equation*}
which is the \emph{Dirac operator} on $Cl_{m+1}$. Analogously, the extension of the slice Cauchy-Riemann $\widetilde{D}_0^{CR}$ operator reads
\begin{equation*}
D_0=e_0\partial_{x_0} + \uo \partial_{r}=e_0 \partial_{x_0} + \frac{\ux}{\vert \ux \vert ^2}\sum_{i=1}^m x_i \pI,
\end{equation*}
which we call the \emph{slice Dirac operator}. Again there is a one to one correspondence between null-solutions of $\widetilde{D}_0^{CR}$ and $D_0$ via \eqref{transfo}. For the rest of the paper, slice monogenic functions are therefore defined to be $Cl_{m+1}$-valued solutions of $D_0$. Together with the multiplication operator $\bfx$, the slice Dirac operator $D_0$ gives a realisation of the $\mathfrak{osp}(1\vert2)$-superalgebra, as will be obtained in Theorem \ref{theorem1}.\\
A general overview of the various Cauchy-Riemann and Dirac operators is given in table \ref{table1}.
\renewcommand\theadset{\renewcommand\arraystretch{1}\setlength\extrarowheight{1pt}}
\setlength{\extrarowheight}{5pt}
\begin{table}[h]
\begin{tabular}{m{0.45\linewidth}m{0.45\linewidth}}
\hline \multicolumn{2}{c}{\thead{\textbf{Algebra}} }  \\ \hline
\thead{$Cl_{m}$} & \thead{$Cl_{m+1}$} \\
\hline \multicolumn{2}{c}{\thead{\textbf{Object}} }\\ \hline
\thead{paravector\\ $x_0+\widetilde{\ux}$} & \thead{$1$-vector\\ $\bfx=x_0e_0+\ux$} \\
\hline \multicolumn{2}{c}{\thead{\textbf{First Cauchy-Riemann generalisation} }}\\ \hline
\thead{Cauchy-Riemann operator\\ $\widetilde{D}^{CR}=\partial_{x_0}+\partial_{\widetilde{\ux}}$} & \thead{Dirac operator\\ $D=e_0\partial_{x_0}+\upx$} \\
\hline \multicolumn{2}{c}{\thead{\textbf{Second Cauchy-Riemann generalisation} }}\\ \hline
\thead{slice Cauchy-Riemann operator\\ $\widetilde{D}_0^{CR}=\partial_{x_0}+\frac{\widetilde{\ux}}{|\widetilde{\ux}|^2}\mE_m$} & \thead{slice Dirac operator\\ $D_0=e_0\partial_{x_0}+\frac{\ux}{|\ux|^2}\mE$}
\end{tabular}
\caption{Overview of the various Cauchy-Riemann and Dirac operators.\label{table1}}
\end{table}
Throughout the article, the notation $f(\bfx)$ is used for the function $f(x_0,\ldots,x_m)$ depending on the separate coordinates $(x_0,\ldots,x_m)$. Analogously $g(\uo)$ is used for the function $g(\omega_1,\ldots,\omega_m)$.
The paper is organised as follows. In section 2 we demonstrate that the definitions of $\bfx$, $D_0$ and $\mE$ give a realisation of the $\mathfrak{osp}(1\vert2)$-superalgebra. In section 3, a Hilbert module is constructed such that $D_0$ is self-adjoint and section 4 treats the polynomial null-solutions of this operator. In section 5 we define Clifford-Hermite polynomials in analogy to their classical counterparts and investigate their properties, such as the associated differential equation. In section 6, Clifford-Hermite functions are defined and normalised. As in the classical case, they turn out to be solutions of a scalar differential equation. Conclusions and suggestions for future research can be found in section 7.
%
\section{The $\mathfrak{osp}(1|2)$-superalgebra}\label{osp}
\setcounter{equation}{0}
As stated in the introduction, the multiplication operator $\bfx$ together with the slice Dirac operator $D_0$ and the full Euler operator $\mE=x_0\partial_{x_0}+\mE_m$ exhibit a particular algebraic structure.
\begin{theorem}\label{theorem1} 
The operators $\mathbf{x}$, $D_0$ and $\mE$ constitute a Lie superalgebra, isomorphic with $\mathfrak{osp}(1|2)$, with relations
\begin{center}
\begin{tabular}{rlrl}
\emph{(i)} & $\{\bfx, \bfx \} = -2 \lvert \bfx \lvert^2$ & \emph{(ii)} & $ \{D_0, D_0 \} = -2 (\partial_{x_0}^2 + \pR^2)$\\
\emph{(iii)} & $\{\bfx, D_0 \} = -2 \left(\mE + 1\right)$ & \emph{(iv)} & $ [\mE + 1 , D_0 ] = -D_0$\\
\emph{(v)} & $[\lvert \bfx\lvert^2, D_0] = -2 \bfx$ & \emph{(vi)} & $ [\mE + 1, \bfx ] = \bfx$\\
\emph{(vii)} & $[\partial_{x_0}^2 + \pR^2, \bfx] = 2 D_0$ & \emph{(viii)} & $ [\mE + 1 , \partial_{x_0}^2 + \pR^2] = -2(\partial_{x_0}^2 + \pR^2)$\\
\emph{(ix)} & $[\partial_{x_0}^2 + \pR^2, \lvert \bfx\lvert^2] = 4 \left(\mE + 1\right)$ & \emph{(x)} & $ [\mE + 1 , \lvert \bfx\lvert^2 ] = 2\lvert \bfx \lvert^2$.
\end{tabular}
\end{center}
\end{theorem}
\begin{proof}
The $\mathfrak{osp}(1\vert2)$-relations can be proven in the slightly more general context of the operators
\begin{alignat*}{2}
\bfx & = \sum_{i=1}^{p}e_i x_i + \sum_{i=p+1}^{p+q}e_i x_i && = \bfx_p + \mathbf{\underline{x}}_q\\
\mE & = \sum_{i=1}^{p}x_i \pI + \sum_{i=p+1}^{p+q}x_i \pI && = \mE_p + \mE_q \\
\uD & = \sum_{i=1}^pe_i\pI + \frac{\mathbf{\ux}_q}{\lvert \mathbf{\ux}_q \lvert^2}\mE_q,&&
\end{alignat*}
with $p, q \in \mathbb{N}$. The case $(p,q)=(1,m)$ corresponds to the statement of the theorem.
The proof is facilitated by performing a radial coordinate transformation on $\bfx_q$ so it can be written as $r\uo$ and the associated partial derivative $\frac{\mathbf{\ux}_q}{\lvert \mathbf{\ux}_q \lvert^2}\mE_q$ equals $\uo \partial_r$.
Here we only prove the relations that are used further on in this paper:
\begin{enumerate} \renewcommand{\labelenumi}{(\roman{enumi})}
\setcounter{enumi}{2}
\item $\begin{aligned}[t]
\{\bfx,\uD\} &= \bfx \uD + \uD \bfx \\
&= \displaystyle 2 \bfx_p \sum_{i=1}^p e_i \pI -2r\pR + 2 \bfx_q \sum_{i=1}^p e_i \pI - 2 r\pR -p-1\\
&= -2 \left( \mE + \frac{p+1}{2}\right),
\end{aligned}$
\item $\begin{aligned}[t]
&\left[ \mE + \frac{p+1}{2},\uD \right] = \mE\uD + \uD \mE\\
&=\displaystyle \mE_p \sum_{i=1}^p e_i \pI - \sum_{i=1}^p e_i \pI + r\pR \uo \pR + \sum_{i=1}^p \sum_{j=1}^p e_0 x_j \pI \partial_{x_j} - \uo\pR -\uo r\pR^2\\
&=-\uD,
\end{aligned}$
\setcounter{enumi}{5}
\item $\begin{aligned}[t]
\left[ \mE , \bfx \right] &= \left[ \sum_{i=1}^p x_i \pI , \bfx_p \right] + \left[ r\pR , \uo r \right] \\
&= \bfx_p(1 + \mE_p) - \bfx_p \mE_p+ \uo r(1+r \pR)-\uo r^2 \pR\\
& = \bfx.
\end{aligned}$
\end{enumerate}
This proves relations (iii), (iv) and (vi) of the theorem. The others are obtained similarly.
\end{proof}
We end this section with a closer look at the action of the differential operator $D_0$ and the Euler operator $\mE$ on $\bfx^\ell,\ \ell \in \mathbb{N}\setminus\{0\}$. The resulting relations will be used later on.
\begin{lemma} \label{Drelaties} One has, with $s \in \mathbb{N}\setminus\{0\}$, the following operator identities:
\begin{align*}
D_0 \bfx^{2s} &=[-2s \bfx^{2s-1} +\bfx^{2s} D_0]\\
D_0 \bfx^{2s+1} &= [-2\bfx^{2s} (s+\mE+1) - \bfx^{2s+1} D_0]
\intertext{and}
\mE \bfx^{2s} &= \bfx^{2s} [2s  + \mE]\\
\mE \bfx^{2s+1} &= \bfx^{2s+1} [2s+1 + \mE].
\end{align*}
\end{lemma}
\begin{proof}
Using the $\mathfrak{osp}(1|2)$-relations in Theorem \ref{theorem1}, we find
\begin{align*}
D_0 \bfx^{2s} &= (-2\mE-2 -\bfx D_0) x^{2s-1}\\
&= \left(-2 (\bfx+\bfx\mE) -2\bfx -\bfx(-2\mE -2 -\bfx D_0) \right) \bfx^{2s-2}\\
&= (-2\bfx + \bfx^2 D_0)\bfx^{2s-2}.
\end{align*}
Repeating this procedure $s$ times, we get
\begin{align*}
D_0 \bfx^{2s}=[-2s \bfx^{2s-1} +\bfx^{2s} D_0]
\end{align*}
for even powers of $\bfx$ and thus
\begin{align*}
D_0 \bfx^{2s+1} &= [-2s \bfx^{2s-1} +\bfx^{2s} D_0] \bfx\\
&= [-2s \bfx^{2s} + \bfx^{2s} (-\bfx D_0 -2 \mE -2) ] \\
&=[-2\bfx^{2s} (s+\mE+1) - \bfx^{2s+1} D_0]
\end{align*}
for odd powers of $\bfx$. Similar calculations can be done for the Euler operator $\mE$:
\begin{align*}
\mE \bfx^{2s} &= (\bfx+\bfx\mE) \bfx^{2s-1} \\
&= \left(\bfx^2 + \bfx(\bfx+\bfx\mE)\right) \bfx^{2s-2} \\
&= \bfx^2 (2+\mE)\bfx^{2s-2}.
\end{align*}
Repeating this procedure $s$ times, we get
\begin{align*}
\mE \bfx^{2s} = \bfx^{2s} [2s  + \mE]
\end{align*}
for even powers of $\bfx$ and
\begin{align*}
\mE \bfx^{2s+1} &= \bfx^{2s} [2s + \mE] \bfx \\
&= \bfx^{2s} [2s \bfx +  (\bfx+\bfx\mE)] \\
&= \bfx^{2s+1} [2s+1 + \mE]
\end{align*}
for odd powers of $\bfx$. 
\end{proof}
%
%
%
\section{Hilbert module}
\label{hilbert}
\setcounter{equation}{0}
The $\mathfrak{osp}(1|2)$-relations allow us to construct a Hilbert module $H$ corresponding to the differential operator $D_0$. We define $H$ as the right $Cl_{m+1}$-module over a weighted $L_2$ function space, where the weight $h$ is still to be determined.
The associated inner product $\langle\ .\ ,\ .\ \rangle : H \times H \rightarrow Cl_{m+1}$ can be written as $\langle f,g \rangle = \int_{\mR^{m+1}}\overline{f}\ g\ h\ \mathrm{d}\bfx$ with $\overline{f}$ the conjugate of the function $f$. In $Cl_{m+1}$, this conjugation is completely defined by its action on the basis vectors:
\begin{align*}
\overline{e_i e_j}&=\overline{e_j}\ \overline{e_i}\\
\overline{e_i}&=-e_i, \qquad i,j=0,\ldots,m.
\end{align*}
The weight function $h$ will be determined by the condition that $D_0$ is self-adjoint with respect to this inner product. We thus demand that
\begin{equation}
\langle D_0 f, g\rangle = \langle f, D_0 g\rangle \label{condition}
\end{equation}
for every $f,g \in \{f: \mR^{m+1} \rightarrow Cl_{m+1} \}$ of suitable decay such that the resulting integrals exist. The left-hand side of condition \eqref{condition} gives
\begin{align} \label{laatsteterm}
&\int_{\mR^{m+1}} \overline{\left( e_0 \partial_{x_0} f + \frac{\ux}{r^2} \sum_{i=1}^m x_i\pI f \right)} \ g\ h\ \mathrm{d}\bfx \nonumber \\
&= - \int_{\mR^m}\left[\overline{f} e_0 g h \right]_{-\infty}^{+\infty} \mathrm{d}\ux + \int_{\mR^{m+1}} \overline{f} e_0  \partial_{x_0}(gh)\ \mathrm{d}\bfx - \sum_{i=1}^m \sum_{j=1}^m \int_{\mR^{m+1}} (\pI \overline{f}) \frac{x_i x_j}{r^2} e_jg h\ \mathrm{d}\bfx \nonumber \\
&= \int_{\mR^{m+1}} \overline{f} e_0  \partial_{x_0}(gh)\ \mathrm{d}\bfx - \sum_{i=1}^m \sum_{j=1}^m \int_{\mR^{m+1}} (\pI \overline{f}) \frac{x_i x_j}{r^2} e_jg h\ \mathrm{d}\bfx.
\end{align}
By partial integration, the terms in the summation can be written as
\begin{align*}
\int_{\mR^{m+1}} (\pI\overline{f}) \frac{x_j x_i}{r^2} e_jg h\mathrm{d}\bfx &= \int_{\mR^{m+1}} \pI \left(\overline{f} \frac{x_j x_i}{r^2} e_j gh \right) \mathrm{d}\bfx- I_{ij}
\end{align*}
with 
\begin{align*}
I_{ij}=\int_{\mR^{m+1}} \overline{f} \pI\left( \frac{x_j x_i}{r^2} e_jg h \right)\mathrm{d}\bfx.
\end{align*}
Because
\begin{equation*}
\pI\left(\frac{x_j x_i}{r^2} g h\right) = \frac{x_j}{r^2} g h + \delta_{ij} \frac{x_i}{r^2} g h - \frac{2x_jx_i^2}{r^4} g h + \frac{x_j x_i}{r^2} \pI (g h), 
\end{equation*}
we find
\begin{align*}
\sum_{i,j=1}^m I_{ij} =& (1- m) \int_{\mR^{m+1}} \overline{f} \frac{\ux}{r^2} g h \mathrm{d}\bfx - \int_{\mR^{m+1}} \overline{f} \frac{1}{r^2} \ux \sum_{i=1}^n x_i\pI (gh) \mathrm{d}\bfx.
\end{align*}
Identifying the right-hand side of \eqref{condition},
\begin{align*}
&\int_{\mR^{m+1}} \overline{f} \left( \partial_{x_0} g e_0 + \frac{\ux}{r^2} \sum_{i=1}^m x_i\pI g \right) \ h\ \mathrm{d}\bfx,
\end{align*}
with expression \eqref{laatsteterm} and writing $\mE_m=\sum_{i=1}^m x_i\pI$, we get
\begin{align}
\int_{\mR^{m+1}} \overline{f} e_0 g (\partial_{x_0} h)\mathrm{d}\bfx - (1-m) \int_{\mR^{m+1}} \overline{f}\frac{\ux}{r^2}  g h \mathrm{d}\bfx + \int_{\mR^{m+1}} \overline{f} \frac{\ux}{r^2} g \mE_m h \mathrm{d}\bfx = 0. \label{eldjkqf}
\end{align}
This must hold for all Clifford-valued functions $f$ and $g$ in the Hilbert module $H$. Therefore, substituting $e_0f$ for $f$ and $e_0g$ for $g$ one obtains
\begin{align*}
\int_{\mR^{m+1}} \overline{f} e_0 g (\partial_{x_0} h)\mathrm{d}\bfx + (1-m) \int_{\mR^{m+1}} \overline{f}\frac{\ux}{r^2}  g h \mathrm{d}\bfx - \int_{\mR^{m+1}} \overline{f} \frac{\ux}{r^2} g \mE_m h \mathrm{d}\bfx = 0,
\end{align*}
again for all Clifford-valued functions $f$ and $g$ in the Hilbert module $H$. Because the sum of these two expressions reads $\int_{\mR^{m+1}} \overline{f} e_0 g (\partial_{x_0} h)\mathrm{d}\bfx=0$, $\partial_{x_0}h$ has to be identically zero. We may thus conclude that $h$ is independent of $x_0$. Equation  \eqref{eldjkqf} simplifies to
\begin{align*}
\int_{\mR^{m+1}} \overline{f}\frac{\ux}{r^2}  g \left[ (m-1) + \mE_m \right] h \mathrm{d}\bfx=0
\end{align*}
for all Clifford-valued functions $f$ and $g$ in $H$, so there must hold that $\mE_m h = (1-m) h = r\partial_r h(r,\uo)$. 
Therefore $h(r,\uo)=r^{1-m}h_2(\uo)$ for some function $h_2(\uo)$. There are no conditions on $h_2$ so we can put $h_2(\uo)=1$. 
The specific form of the weight function $h(r)=r^{1-m}$ simplifies the calculation of the inner product thoroughly: because $\mathrm{d}\bfx = \mathrm{d}x_0 r^{m-1} \mathrm{d}r \mathrm{d} \uo$, one gets
\begin{equation*}
\langle f,g \rangle = \int_{\mR^{m+1}}\overline{f}\ g\ \mathrm{d}x_0 \mathrm{d}r \mathrm{d} \uo,
\end{equation*}
which can be seen as a cartesian integration in the coordinate $x_0$ as well as in the spherical coordinate $r$. The Hilbert module associated with $D_0$ is thus given by
\begin{align*}
\mathcal{L}_2 &= L_{2}(\mR^{m+1},r^{1-m}\mathrm{d}\bfx)\ \otimes\ Cl_{m+1}\\
&=\left\{ f : \mR^{m+1} \rightarrow Cl_{m+1}\ \vline\ \left[ \int_{\mR^{m+1}} \overline{f(\bfx)} f(\bfx)\ r^{1-m}\ \mathrm{d}\bfx \right]_0 < +\infty  \right\},
\end{align*}
where the notation $[\ldots]_0$ denotes the scalar part of the expression between the brackets. We have hence obtained the following result.
\begin{proposition}\label{innerprod}
The inner product $\langle f,g \rangle = \int_{\mR^{m+1}}\overline{f}g\ \mathrm{d}x_0 \mathrm{d}r \mathrm{d} \uo$ on the right $Cl_{m+1}$-module $\mathcal{L}_2$ exhibits the relations
\begin{align*}
\langle D_0 f, g\rangle &= \langle f, D_0 g\rangle,\\
\langle \bfx f, g\rangle &= - \langle f, \bfx g\rangle
\end{align*} on a dense subset of $\mathcal{L}_2$.
\end{proposition}
\section{Polynomials in the kernel of $D_0$}
\setcounter{equation}{0}
\label{poly}
Given the differential operator $D_0$ and the above defined inner product, it should be possible to construct a Clifford analogue to the classical Hermite functions. However, before doing so we have to determine the kernel of $D_0$ on which these Clifford-Hermite functions will be based.
\begin{proposition}
The homogeneous polynomials
\begin{equation*}
m_k(\bfx)=\left(e_0-1\right) (x_0+\ux)^k \mathbf{a}
\end{equation*}
of degree $k \in \mathbb{N}$ with $\mathbf{a} \in Cl_{m+1}$ vanish under the action of $D_0$.
\end{proposition}
\begin{proof}
A straightforward calculation of $D_0 m_k(\bfx)$ gives zero.
\end{proof}
Given that these polynomials $m_k$ are in the kernel of $D_0$, the question naturally arises whether they yield all polynomial null-solutions of $D_0$.
A general Clifford-valued polynomial $f: \mR^{m+1} \rightarrow Cl_{m+1}: \bfx \mapsto f(\bfx)$ can be decomposed into a Taylor series in the coordinate $x_0$. The coefficients in this decomposition are polynomials $p: \mR^m \rightarrow Cl_{m+1}$. Because $D_0$ maps $k$-homogeneous polynomials to homogeneous polynomials of degree $k-1$, we will seek its null-solutions within the space of $k-$homogeneous polynomials $\cP_k(\mR^m, Cl_{m+1})$.
The $k$-homogeneous part of $f$ is given by
\begin{align*}
f_k(\bfx)&=\sum_{i=0}^{k} x_0^i p_{k-i}(\ux)\\
&=\sum_{i=0}^{k} x_0^i r^{k-i} p_{k-i}(\uo).
\end{align*}
with $p_{k-i}:\mR^m \rightarrow Cl_{m+1}$ and $\mE p_{k-i} = (k-i) p_{k-i}$. For the function $f_k$ to be in the kernel of $D_0$, we demand that
\begin{align*}
D_0 f_k(\bfx)=e_0 \sum_{i=0}^{k-1} (i+1) x_0^i r^{k-i-1} p_{k-i-1}(\uo) + \uo \sum_{i=0}^{k-1} x_0^i (k-i) r^{k-i-1} p_{k-i}(\uo) = 0
\end{align*}
and therefore
\begin{align*}
p_{k-i}(\uo) &= \uo e_0 \frac{(i+1)}{(k-i)} p_{k-i-1}(\uo)\\
&=\frac{(i+1)}{(k-i)} \frac{(i+2)}{(k-i-1)} \frac{\ldots}{\ldots} \frac{k}{1} (\uo e_0)^{k-i} p_{0}(\uo)\\
&=\binom{k}{i} (\uo e_0)^{k-i} p_{0}(\uo)
\end{align*}
for $i\in \{0,\ldots,k-1\}$ so
\begin{align*}
f_k(\bfx)=\sum_{i=0}^{k} \binom{k}{i} x_0^i r^{k-i} (\uo e_0)^{k-i} p_{0}(\uo)=(x_0+\ux e_0)^k p_{0}(\uo).
\end{align*}
The polynomial $p_{0} \in \mathcal{P}_0(\mR^m, Cl_{m+1})$ is a constant so it can be written as $p_{0}=(e_0-1)\mathbf{a}$ with $\mathbf{a}\in Cl_{m+1}$. The general expression for  a monogenic of degree $k$ thus reads 
\begin{align*}
f_k(\bfx)&=(x_0+\ux e_0)^k (e_0-1)\mathbf{a}\\
&=(e_0-1) (x_0+\ux)^k \mathbf{a}
\end{align*}
and we have proven the following theorem.
\begin{theorem}
The set of $k$-homogeneous polynomials $(k \in \mathbb{N})$ in the kernel of $D_0$ is one-dimensional and is given by the polynomials
\begin{equation*}
m_k(\bfx)=\left(e_0-1\right) (x_0+\ux)^k\ \mathbf{a}, \qquad \mathbf{a} \in Cl_{m+1}.
\end{equation*}
\end{theorem}
%
\section{Clifford-Hermite polynomials}
\setcounter{equation}{0}
\label{clifherm}
In classical Fourier analysis, the eigenvectors of the Fourier operator are given by the Hermite functions. Together with their associated Hermite polynomials they also play an important role in physics, for instance as solutions of the quantum harmonic oscillator. Various generalisations of the Hermite polynomials have been studied in Clifford analysis \cite{de2007hermite,delanghe1992clifford}. Following their approach we may define Hermite polynomials for the slice Dirac operator as 
$h_{j,k}(\bfx)m_k(\bfx)=(\bfx-cD_0)^j m_k(\bfx)$ with $c\in \mathbb{C}$ a complex parameter. Because we don't want to overload notations, the presence of this parameter will not be stressed in the various definitions and notations. The first five Clifford-Hermite polynomials are then
\begin{align}
h_{0,k}(\bfx) &= 1 \nonumber\\
h_{1,k}(\bfx) &= \bfx \nonumber\\
h_{2,k}(\bfx) &= [\bfx^2 + 2c(k+1)] \label{vijf} \\
h_{3,k}(\bfx) &= [\bfx^3 + 2c(k+2)\bfx] \nonumber\\
h_{4,k}(\bfx) &= [\bfx^4 + 4c(k+2)\bfx^2 + 4c^2(k+1)(k+2)].\nonumber
\end{align}
In what follows, the polynomial product of the Clifford-Hermite polynomial $h_{j,k}$ and a $k$-homogeneous monogenic polynomial $m_k$ will be denoted as $H_j (m_k)$, so $H_j (m_k)(\bfx)  = h_{j,k}(\bfx)m_k(\bfx) = (\bfx-cD_0)^j m_k(\bfx)$. 
\begin{proposition}
(recursion formula) For every $j \in \mathbb{N} \setminus\{0\}$ and $c\in \mathbb{C}$ one has
\begin{equation}
H_{j}(m_k)(\bfx) = (\bfx-cD_0) H_{j-1}(m_k)(\bfx). \label{recursion}
\end{equation}
\end{proposition}
\begin{proof}
By definition.
\end{proof}
As is the case with the classical Hermite polynomials, the polynomials $H_j(m_k)$ are solutions of a partial differential equation. To prove this statement, we need the two following lemmata.
\begin{lemma} \label{lemma}
The $\mathfrak{osp}(1\vert2)$-superalgebra exhibits the commutation relation
\begin{equation*}[\mE + \bfx D_0,(\bfx-cD_0)^{2}]=0\end{equation*} where $c\in \mathbb{C}$.
\end{lemma}
\begin{proof}
Using the $\mathfrak{osp}(1\vert2)$-relations in Theorem \ref{theorem1}, one obtains the operator identity $(\mE + \bfx D_0) (\bfx - cD_0) = - (\bfx - cD_0) (\mE + \bfx D_0 + 1)$ and thus
\begin{align*}
(\mE + \bfx D_0)(\bfx-cD_0)^{2} &= -(\bfx-cD_0)[ - (\bfx - cD_0) (\mE + \bfx D_0 + 1) + (\bfx-cD_0)]\\
 &=(\bfx-cD_0)^2(\mE+\bfx D_0),\end{align*}
which proves the lemma. \end{proof}
\noindent As a consequence, the following properties hold for the polynomials $H_j(m_k)$.
\begin{lemma} \label{evenodd} One has
\begin{align*}
(\mE + \bfx D_0)H_{j} (m_k)(\bfx) &= B(j,k) H_{j} (m_k)(\bfx)
\end{align*}
with $B(j,k)=k$ if $j=2t$ and $B(j,k)=-(k+1)$ if $j=2t+1$.
\end{lemma}
\begin{proof} Using the definition of $H_j(m_k)$, we get
\begin{equation*}
(\mE + \bfx D_0)H_{2t} (m_k)(\bfx) = (\bfx-cD_0)^{2t} (\mE+\bfx D_0)m_k(\bfx) = k H_{2t} (m_k)(\bfx)
\end{equation*}
for polynomials $H_j(m_k)$ of even order $j=2t$ and
\begin{align*}
(\mE + \bfx D_0) H_{2t+1} (m_k)(\bfx) &= (\bfx-cD_0)^{2t} (\mE +\bfx D_0)(\bfx-cD_0) m_k(\bfx) \nonumber \\
&= (\bfx-cD_0)^{2t} [-(\bfx-cD_0)(\mE+\bfx D_0)- (\bfx-cD_0)] m_k(\bfx) \nonumber \\
&= (\bfx-cD_0)^{2t} [-(k + 1) (\bfx-cD_0)] m_k(\bfx) \nonumber \\
&= -(k+1)H_{2t+1} (m_k)(\bfx)
\end{align*}
for polynomials $H_j(m_k)$ of odd order $j=2t+1$.
\end{proof}
Now we have all ingredients to obtain a partial differential equation for the polynomials $H_j(m_k)$. Note that this is not a scalar partial differential equation.
\begin{theorem} \label{differential}
The polynomials $H_j(m_k)(\bfx)=h_{j,k}(\bfx)m_k(\bfx)$, with $h_{j,k}$ the Clifford-Hermite polynomial of degree $j$, are solutions of the differential equation
\begin{align*}
cD_0^2 H_{j}(m_k)(\bfx) -\bfx D_0 H_{j}(m_k)(\bfx)+ C(j,k) H_{j}(m_k)(\bfx) =0
\end{align*}
with $C(j,k)=-2t$ if $j=2t$ and $C(j,k)=-2(k+t+1)$ if $j=2t+1$.
\end{theorem}
\begin{proof}
We prove the identities
\begin{align}
D_0H_{2t}(m_k)(\bfx) &= -2t H_{2t-1}(m_k)(\bfx),\label{id1} \\
D_0H_{2t+1}(m_k)(\bfx) &= -2(k+t+1) H_{2t}(m_k)(\bfx)\label{id2},
\end{align}
from which the theorem immediately follows under the action of $\bfx-cD_0$.
We will prove these relations by using the operator relation $\{\bfx, D_0 \} = -2 (\mE + 1)$ and Lemma \ref{evenodd}.\\
For $j=0$ equation \eqref{id1} is trivial. In the case of even degree $j=2t, t \neq 0,$ one has
\begin{align*}
D_0H_{2t}(m_k)(\bfx) 
&= D_{0} (\bfx-cD_0)H_{2t-1}(m_k)(\bfx)\\
&= (-cD_0^2-\bfx D_0-2\mE-2)H_{2t-1}(m_k)(\bfx)\\
&= \left[(\bfx-cD_0)D_0 - 2(\bfx D_0+\mE) -2 \right]H_{2t-1}(m_k)(\bfx)\\
&= \left[(\bfx-cD_0)D_0 + 2(k+1) -2 \right]H_{2t-1}(m_k)(\bfx)\\
&= (\bfx-cD_0)\left[(\bfx-cD_0)D_0 - 2(\bfx D_0+\mE) -2 \right]H_{2t-2}(m_k)(\bfx)\\
& \phantom{=(} + 2k H_{2t-1}(m_k)(\bfx)\\
&= (\bfx-cD_0)^2D_0H_{2t-2}(m_k)(\bfx) + (2k -2k -2) H_{2t-1}(m_k)(\bfx)\\
&= (\bfx-cD_0)^2D_0H_{2t-2}(m_k)(\bfx) - 2 H_{2t-1}(m_k)(\bfx).
\end{align*}
Repeating this procedure $t$ times, we get
\begin{align*}
D_0H_{2t}(m_k)(\bfx) &=-2tH_{2t-1}(m_k)(\bfx) + (\bfx-cD_0)^{2t}D_0 m_k(\bfx) \\
&=-2tH_{2t-1}(m_k)(\bfx).
\end{align*}
The case of odd $j=2t+1$ follows from the previous result and Lemma \ref{evenodd}:
\begin{align*}
D_0H_{2t+1}(m_k)(\bfx) &= D_{0} (\bfx-cD_0)H_{2t}(m_k)(\bfx)\\
&=((\bfx -c D_0) D_0 - 2(\mE + \bfx D_0) -2 ) H_{2t}(m_k)(\bfx)\\
&=-2(t+k+1) H_{2t}(m_k)(\bfx).\end{align*}
This proves the theorem.\end{proof} 
Equations \eqref{id1} and \eqref{id2} give rise to a second recursion formula.
\begin{proposition}
(recursion formula bis) One has
\begin{align*}
H_{j+1}(m_k) = \bfx H_{j}(m_k) - c\ C(j,k) H_{j-1}(m_k)
\end{align*}
with $C(j,k)$ as in Theorem \ref{differential}.
\end{proposition}
\begin{proof} Using equations \eqref{id1} and \eqref{id2}, one finds immediately that
\begin{align*}H_{j+1}(m_k) &= (\bfx-cD_0) H_{j}(m_k)\\
&=\bfx H_{j}(m_k) - c\ C(j,k) H_{j-1}(m_k).\end{align*}
\end{proof}
In each factor $(\bfx-cD_0)$ of $H_j (m_k)(\bfx)  =  (\bfx-cD_0)^j m_k(\bfx)$, the first term raises and the second term lowers the degree of the polynomial it is acting upon by one. Therefore
the Hermite polynomials $h_{2t,k}$ (respectively $h_{2t+1,k}$) will consist of even (respectively odd) powers of $\bfx$ only and we may write:
\begin{align}
H_{2t}(m_k)(\bfx)&= \sum_{i=0}^s a_{2i}^{2t} \bfx^{2i} m_k(\bfx)\\
H_{2t+1}(m_k)(\bfx) &=  \sum_{i=0}^s a_{2i+1}^{2t+1} \bfx^{2i+1} m_k(\bfx).
\end{align}
Because the polynomials $H_j(m_k)$ satisfy the identities \eqref{id1} and \eqref{id2}, recursion relations can be derived between the coefficients $a_{2i}^{2t}$ and $a_{2i+1}^{2t+1}$ of the Hermite polynomials $h_{j,k}$. In doing so, we will be able to relate these polynomials to generalised Laguerre polynomials. Using Lemma \ref{Drelaties}, one has
\begin{align*}
\sum_{i=0}^{t-1} -2(i+1) a_{2i+2}^{2t}\bfx^{2i+1} m_k(\bfx) &= -2t \sum_{i=0}^{t-1} a_{2i+1}^{2t-1}\bfx^{2i+1}m_k(\bfx)\\
\sum_{i=0}^{t} -2(i+k+1) a_{2i+1}^{2t+1}\bfx^{2i} m_k(\bfx) &= -2(t+k+1) \sum_{i=0}^{t} a_{2i}^{2t}\bfx^{2i} m_k(\bfx)
\end{align*}
so
\begin{align*}
\phantom{(+\alpha+1)}i\ a_{2i}^{2t} &= t\ a_{2i-1}^{2t-1} \qquad \qquad \\
(i+k+1) \ a_{2i+1}^{2t+1} &= (t+k+1) \ a_{2i}^{2t\phantom{-1}} \qquad
\end{align*}
for $i \in \{0,\ldots,t\}$ and the following recursion relations are found:
\begin{align*}
a_{2i}^{2t} &= \frac{t}{i}\frac{t+k}{i+k}a_{2i-2}^{2t-2} \\
&=\frac{t!}{(t-i)!i!}\frac{(t+k)!}{(t+k-i)!}\frac{k!}{(i+k)!}a^{2t-2i}_0
\intertext{and}
a_{2i+1}^{2t+1} &= \frac{t+k+1}{i+k+1} \frac{t}{i} a_{2i-1}^{2t-1} \\
&=\frac{(t+k+1)!}{(t+k-i+1)!}\frac{(k+1)!}{(i+k+1)!}\frac{t!}{(t-i)!i!}a^{2t-2i+1}_1.
\end{align*}
Putting $\bfx=0$ in the expression $H_j(m_k)(\bfx)=(\bfx-cD_0)H_{j-1}(m_k)(\bfx)$ for $j=2t$ gives $a_0^{2t}=2c(k+1)a_1^{2t-1}$ so we find that 
$a^{2t}_{0} = (2c)^t\frac{(t+k)!}{k!} a_0^0$.
Because $H_0(m_k)(\bfx)=a_0^0m_k(\bfx)=m_k(\bfx)$ we conclude that $a^{2t}_{0} = (2c)^t\frac{(t+k)!}{k!}$ and $a^{2t+1}_{1} = (2c)^t\frac{(t+k+1)!}{(k+1)!}$. The final expressions for the coefficients are thus given by
\begin{align*}
a_{2i}^{2t} &=\binom{t}{i} (2c)^{t-i} \frac{(t+k)!}{(k+i)!} \\
a_{2i+1}^{2t+1}&=\binom{t}{i} (2c)^{t-i} \frac{(t+k+1)!}{(k+i+1)!}.
\end{align*}
This result allows to rewrite the Hermite polynomials as stated in the following theorem.
\begin{theorem}
The Hermite polynomials $h_{j,k}$ can be expressed as
\begin{align*}
h_{2t,k}(\bfx)&=(2c)^t t!\ L^k_t\left(\frac{\vert\bfx\vert^2}{2c}\right)\\
h_{2t+1,k}(\bfx)&=(2c)^t t!\ \bfx\ L^{k+1}_t\left(\frac{\vert\bfx\vert^2}{2c}\right)
\end{align*}
where $L_t^k$ are the generalised Laguerre polynomials of degree $t$ on the real line \cite{andrews2000special}.
\end{theorem}
\begin{proof}
Because of the preceding considerations, it follows that
\begin{align*}
h_{2t,k}(\bfx)&= \sum_{i=0}^t \binom{t}{i} (2c)^{t-i} \frac{(t+k)!}{(k+i)!} \bfx^{2i}\\
&=(2c)^{t} \sum_{i=0}^t \frac{(-1)^i}{i!} \frac{t!}{(t-i)!} \frac{(t+k)!}{(k+i)!} \left( -\frac{\bfx^{2}}{2c} \right)^i\\
&=(2c)^t t!\ L^k_t\left(\frac{\vert\bfx\vert^2}{2c}\right)
\end{align*}
and analogously for $h_{2t+1,k}$.
\end{proof}
We conclude this paragraph by considering the Rodrigues formula. In order to do so, we first define the generalized Gaussian function
\begin{align}
\exp(-\lvert \bfx \lvert^2) = \sum_{s=0}^\infty \frac{(-1)^s}{s!} \lvert\bfx\lvert^{2s}.
\end{align}
We get the following theorem.
\begin{theorem}
(Rodrigues formula) The polynomials $H_j(m_k)$ can be written as 
\begin{align*}
H_j(m_k)(\bfx)&= \exp\left(\frac{\lvert\bfx\lvert^2}{\alpha}\right) \left[\left(1-\frac{2c}{\alpha}\right) \bfx -cD_0\right]^j \exp\left(-\frac{\lvert\bfx\lvert^2}{\alpha}\right)m_k(\bfx)
\end{align*} with $\alpha \in \mathbb{R}^+_0$.
\end{theorem}
\begin{proof}
Using Lemma \ref{Drelaties}, one has
\begin{align}
\exp\left(\frac{\lvert\bfx\lvert^2}{\alpha}\right) D_0 \exp\left(-\frac{\lvert\bfx\lvert^2}{\alpha}\right)&= \exp\left(\frac{\lvert\bfx\lvert^2}{\alpha}\right) \sum_{i=0}^\infty \frac{1}{i!} \left(\frac{1}{\alpha} \right)^i D_0 \bfx^{2i}\nonumber\\
&= \exp\left(\frac{\lvert\bfx\lvert^2}{\alpha}\right) \sum_{i=0}^\infty \frac{1}{i!}\left(\frac{1}{\alpha} \right)^i (-2i \bfx^{2i-1} + \bfx^{2i} D_0)\nonumber\\
&= \left(-\frac{2}{\alpha}\bfx+D_0\right) \label{alpha}.
\end{align}
Because $\bfx$ commutes with $\exp\left(-\frac{\lvert\bfx\lvert^2}{\alpha}\right)$, it follows that
\begin{align}\label{alphas}
\exp\left(\frac{\lvert\bfx\lvert^2}{\alpha}\right) \left[\left(1-\frac{2c}{\alpha}\right) \bfx -cD_0\right] \exp\left(-\frac{\lvert\bfx\lvert^2}{\alpha}\right) = \left(\bfx-cD_0\right).
\end{align}
This proves the theorem.
\end{proof}
%
%
\section{Clifford-Hermite functions}
\setcounter{equation}{0}
Based on the definition of the classical Hermite polynomials, we defined the Clifford-Hermite polynomials $h_{j,k}$ as $h_{j,k}(\bfx)m_k(\bfx)=(\bfx-cD_0)^j m_k(\bfx)$. The analogy is pushed further by introducing Clifford-Hermite functions $\psi_{j,k}$ as a product of the corresponding polynomial $h_{j,k}$ and an exponential function: 
\begin{equation}\label{CHdef}
\psi^\beta_{j,k}(\bfx)=h_{j,k}(\bfx)m_k(\bfx)\exp\left(-\frac{\vert \bfx \vert^2}{\beta}\right),
\end{equation}
with $\beta\in\mR^+_0$ to be determined. The parameter $\beta$ will be fixed such that the resulting Clifford-Hermite functions $\psi_{j,k}$ are orthogonal with respect to the inner product in Proposition \ref{innerprod}. 
\begin{proposition}
With respect to the inner product in Proposition \ref{innerprod}, the functions 
\begin{equation*}
\psi^\beta_{j,k}(\bfx)=h_{j,k}(\bfx)m_k(\bfx)\exp\left(-\frac{\vert \bfx \vert^2}{\beta}\right),
\end{equation*}
$\beta\in\mR^+_0$, can only be orthogonal if $\beta=4c$.
\end{proposition}
\begin{proof}
The orthogonality of the functions $\psi^{\beta}_{j,k}$ implies that $\langle \psi^{\beta}_{j_1,k_1} \vert \psi^{\beta}_{j_2,k_2} \rangle=0$ for all $j_1,j_2,k_1,k_2 \in \mathbb{N} \cup \{0\}$ with $j_1\not=j_2$ or $k_1\not= k_2$. Based on equation \eqref{vijf} and definition \ref{CHdef}, we write $\psi^{\beta}_{0,0}=\mathbf{a_1}\exp{(-\lvert \bfx \lvert^2/\beta)}$ and $\psi^{\beta}_{2,0}=(\bfx^2+2c)\mathbf{a_2}\exp{(-\lvert \bfx \lvert^2/\beta)}$ with $\mathbf{a_1},\mathbf{a_2} \in Cl_{m+1}$. The inner product of $\psi^{\beta}_{0,0}$ and $\psi^{\beta}_{2,0}$ is then
\begin{align*}
\langle \psi^{\beta}_{0,0}, \psi^{\beta}_{2,0} \rangle &= \overline{\mathbf{a}}_{\mathbf{1}} \int_{\mathbb{R}^{m+1}} (\bfx^2+2c) \exp{\left(-2\frac{\lvert \bfx \lvert^2}{\beta}\right)} r^{1-m} \mathrm{d}\bfx\ \mathbf{a_2}\\
&= \overline{\mathbf{a}}_{\mathbf{1}} \int_{-\infty}^{+\infty} \int_{0}^{+\infty} \int_{\mathbb{S}^{m-1}} (-x_0^2 - r^2+2c) \exp{\left(-2\frac{x_0^2+r^2}{\beta}\right)} \mathrm{d}x_0 \mathrm{d}r \mathrm{d}\uo\ \mathbf{a_2}.
\end{align*}
Because $\int_{-\infty}^{+\infty} \exp(-x^2) \mathrm{d}x = \sqrt{\pi}$, the integral yields
\begin{align*}
\langle \psi^{\beta}_{0,0}, \psi^{\beta}_{2,0} \rangle = \overline{\mathbf{a}}_{\mathbf{1}}\mathbf{a_2} \left(4c - \beta \right) \left(\frac{\pi\beta}{8} \right) \int_{\mathbb{S}^{m-1}} \mathrm{d}\uo.
\end{align*}
The remaining integral gives the surface of an $(m-1)$-dimensional sphere, which is a real number. The expression will thus only be zero if $\beta=4c$.
\end{proof}
We thus define the Clifford-Hermite functions as 
\begin{equation*}
\psi_{j,k}(\bfx)=h_{j,k}(\bfx)m_k(\bfx)\exp\left(-\frac{\vert \bfx \vert^2}{4c}\right).
\end{equation*}
For this particular value of $\beta$, these functions establish a remarkable property.
\begin{proposition}
The Clifford-Hermite functions $\psi_{j,k}$ satisfy the relations
\begin{align}\label{psis}
\begin{cases}
\phantom{\tilde{D_c}^\dagger}\psi_{j,k}&=\tilde{D_c}\psi_{j-1,k}\\
\tilde{D_c}^\dagger\psi_{j,k}&= -c\ C(j,k) \psi_{j-1,k}
\end{cases}
\end{align}
with $\tilde{D_c}=\frac{\bfx}{2} -cD_0$ and $C(j,k)$ as in Theorem \ref{differential}. The dagger denotes the adjoint with respect to the inner product defined in Proposition \ref{innerprod}.
\end{proposition}
\begin{proof}
Based on equation \eqref{alphas}, one has
\begin{align*}
\psi_{j,k}(\bfx)&=\exp\left(-\frac{\vert \bfx \vert^2}{4c}\right)\left(\bfx -cD_0 \right) h_{j-1,k}(\bfx)m_k(\bfx)\\
&=\left(\frac{\bfx}{2} -cD_0\right) \exp\left(-\frac{\lvert\bfx\lvert^2}{4c}\right) h_{j-1,k}(\bfx)m_k(\bfx)\\
&=\left(\frac{\bfx}{2} -cD_0\right) \psi_{j-1,k}(\bfx).
\intertext{With respect to the inner product in Proposition \ref{innerprod}, the action of the adjoint operator $\tilde{D_c}^\dagger = \left(\frac{\bfx}{2} -cD_0\right)^\dagger$ on $\psi_{j,k}$ is given by}
\left(\tilde{D_c}\right)^\dagger \psi_{j,k}(\bfx)&=\left(-\frac{\bfx}{2} -cD_0\right) \exp\left(-\frac{\lvert\bfx\lvert^2}{4c}\right) h_{j,k}(\bfx)m_k(\bfx)\\
&=\exp\left(-\frac{\lvert\bfx\lvert^2}{4c}\right) \left[-\bfx + \bfx -cD_0 \right] h_{j,k}(\bfx)m_k(\bfx)\\
&=-c \exp\left(-\frac{\lvert\bfx\lvert^2}{4c}\right) D_0 h_{j,k}(\bfx)m_k(\bfx),
\end{align*}
again because of equation \eqref{alphas}. Applying equations \eqref{id1} and \eqref{id2} to the last expression ends the proof.
\end{proof}
The main purpose of this section is to prove the orthogonality of the Clifford-Hermite functions, so we want to calculate the inner product of two functions $\psi_{j_1,k_1}$ and $\psi_{j_2,k_2}$. The above definition of the Clifford-Hermite functions (with $\beta=4c$) and the associated properties \eqref{psis} allow a drastic simplification of this calculation. One has:
\begin{align}
\langle \psi_{j_1,k_1} ,\psi_{j_2,k_2} \rangle &= \langle \tilde{D_c} \psi_{j_1-1,k_1} ,\psi_{j_2,k_2} \rangle \nonumber\\
&=\langle \psi_{j_1-1,k_1} ,\tilde{D_c}^\dagger \psi_{j_2,k_2} \rangle \nonumber\\
&=-c\ C(j_2,k_2) \langle \psi_{j_1-1,k_1}, \psi_{j_2-1,k_2} \rangle \label{proces}.
\end{align}
Continuing this procedure further decreases the degree of both Hermite functions. Before we give the final expression for this inner product, we prove the following lemma:
\begin{lemma} \label{innermonogeen}
One has
\begin{align*}
\left\langle \psi_{0,k_1}, \psi_{0,k_2} \right\rangle&=\left\langle m_{k_1}(\bfx)e^{-\vert \bfx \vert^2/4c}, m_{k_2}(\bfx)e^{-\vert \bfx \vert^2/4c} \right\rangle\\
&=2\overline{\mathbf{a}}_\mathbf{1} \mathbf{a_2} (2c)^{k_1+1} \pi^{m/2+1} \frac{\Gamma(k_1+1)}{\Gamma(m/2)} \delta_{k_1k_2}
\end{align*} where $m_{k_i}(\bfx)=(e_0-1)(x_0+\ux)^{k_i} \mathbf{a_i}$ with $\mathbf{a_i}\in Cl_{m+1}$ and $i=1,2$.
\end{lemma}
\begin{proof}
We will assume that $k_1\leq k_2$, the case $k_1> k_2$ being completely analogous. We have
\begin{align*}
\left\langle m_{k_1}(\bfx)e^{-\vert \bfx \vert^2/4c} \right.&\left., m_{k_2}(\bfx)e^{-\vert \bfx \vert^2/4c} \right\rangle\\
= 2\overline{\mathbf{a}}_\mathbf{1} &\left( \int_{-\infty}^{+\infty} \int_0^{+\infty} \int_{\mS^{m-1}} (x_0-\ux)^{k_1} (x_0+\ux)^{k_2} e^{-\vert \bfx \vert^2/2c} \mathrm{d}x_0 \mathrm{d}r \mathrm{d}\sigma \right) \mathbf{a_2}\\
= 2\overline{\mathbf{a}}_\mathbf{1} &\left( \int_{-\infty}^{+\infty} \int_0^{+\infty} \int_{\mS^{m-1}} (x_0^2+r^2)^k (x_0+\ux)^{K} e^{-\vert \bfx \vert^2/2c} \mathrm{d}x_0 \mathrm{d}r \mathrm{d}\sigma \right) \mathbf{a_2}
\end{align*}
where $k=k_1$ and $K=k_2-k_1$. Going over to polar coordinates $(r,\theta)$ by putting $x_0=\sqrt{2c} R\cos \theta,r= \sqrt{2c} R\sin \theta$ and integrating over the upper half plane, the inner product equals
\begin{align*}
2\overline{\mathbf{a}}_\mathbf{1} (\sqrt{2c})^{2k+K+2}\ \underbrace{\int_0^{+\infty} R^{2k+K+1} e^{\left(-R^2\right)} \mathrm{d}R}_{\textrm{A}}\ \underbrace{\int_0^{\pi} \int_{\mS^{m-1}} (\cos \theta+\sin \theta\ \uo)^{K} \mathrm{d}\theta \mathrm{d}\sigma}_{\textrm{B}}\ \mathbf{a_2}.
\end{align*}
Writing the integrand of B as $\sum_{i=0}^K \binom{K}{i} \cos^i \theta \sin^{K-i}\theta\ \uo^{K-i}$, we can split B into a sum of products of integrals:
\begin{align*}
\int_0^{\pi} \int_{\mS^{m-1}} (\cos \theta+\sin \theta\ \uo)^{K} \mathrm{d}\theta \mathrm{d}\sigma 
= \sum_{i=0}^K \binom{K}{i} \underbrace{\int_0^{\pi} \cos^i \theta \sin^{K-i}\theta\ \mathrm{d}\theta}_{\textrm{B1}}\ \underbrace{\int_{\mS^{m-1}} \uo^{K-i} \mathrm{d}\sigma}_{\textrm{B2}}.
\end{align*}
The integral B1 vanishes for odd powers $i$. For $i=2I$ its value is given by
\begin{align*}
\textrm{B1} = \frac{2I-1}{K-2I+1} \frac{2I-3}{K-2I+3}  \ldots \frac{1}{K-1} 
\begin{cases}
\frac{K-1}{K}\frac{K-3}{K-2} \ldots \frac{1}{2}\ \pi \phantom{2} \qquad K \textrm{ even},\\
\frac{K-1}{K}\frac{K-3}{K-2} \ldots \frac{2}{3}\ 2 \phantom{\pi} \qquad K \textrm{ odd}.
\end{cases}
\end{align*}
In addition, B2 vanishes for odd values of $K-i$. For $K=2L$ and $i=2I$ its value is given by
\begin{align*}
\textrm{B2} = (-1)^{L-I} \frac{2\pi^{m/2}}{\Gamma(m/2)}.
\end{align*}\
The integral B thus simplifies to
\begin{align*}
B&=\pi \frac{2\pi^{m/2}}{\Gamma(m/2)} \sum_{I=0}^L \frac{(2L)!}{(2I)!(2L-2I)!} \frac{(2I)!}{2^I I!} \frac{2^I(2L-2I)!L!}{(2L)!(L-I)!}\frac{(2L)!}{2^{2L}(L!)^2}(-1)^{L-I} \\
&=\frac{1}{2^{2L}} \pi \frac{2\pi^{m/2}}{\Gamma(m/2)} \binom{2L}{L} \sum_{I=0}^L \binom{L}{I}(-1)^{L-I}.
\end{align*}
The summation in the last factor is the binomial expansion of $(1-1)^L$, which is zero unless $L=0$. This implies that $K$ has to be zero, so $k_1$ must equal $k_2$. Using the definition of the gamma function, the inner product subsequently reads
\begin{align*}
\left\langle \psi_{0,k_1}, \psi_{0,k_2} \right\rangle
&= 2\overline{\mathbf{a}}_\mathbf{1} (2c)^{k_1+1} \pi \frac{2\pi^{m/2}}{\Gamma(m/2)} \delta_{k_1k_2} \int_0^{+\infty} R^{2k_1+1} e^{\left(-R^2\right)} \mathrm{d}R\ \mathbf{a_2}\\
&= 2\overline{\mathbf{a}}_\mathbf{1} (2c)^{k_1+1} \frac{2\pi^{m/2+1}}{\Gamma(m/2)} \delta_{k_1k_2} \frac{\Gamma(k_1+1)}{2} \mathbf{a_2},
\end{align*}
which proves the lemma.
\end{proof}
Now we can determine the inner product of two Clifford-Hermite functions $\psi_{j_1,k_1}$ and $\psi_{j_2,k_2}$. Different cases will be distinguished corresponding to the parity of $j_1$ and $j_2$.
\begin{theorem}
Let $\psi_{j_i,k_i}=(\bfx-cD_0)^{j_i}m_{k_i}(\bfx)\exp(-\vert \bfx \vert^2/4c)$ with $m_{k_i}(\bfx)=(e_0-1)(x_0+\ux)^{k_i} \mathbf{a_i}$ and $\mathbf{a_i}\in Cl_{m+1}$ for $i=1,2$.
The inner product of these two Clifford-Hermite functions $\psi_{j_1,k_1}$ and $\psi_{j_2,k_2}$ is given by
\begin{align*}
\left\langle \psi_{j_1,k_1}, \psi_{j_2,k_2} \right\rangle = A(j_1,k_1) \delta_{j_1j_2} \delta_{k_1k_2}
\end{align*}
with
\begin{align*}A(j_1,k_1)=
\begin{cases}
2\overline{\mathbf{a}}_\mathbf{1} \mathbf{a_2}(2c)^{2t_1+k_1+1} t_1! (k_1+t_1)! \frac{\pi^{m/2+1}}{\Gamma(m/2)} & j_1=2t_1,\\
2\overline{\mathbf{a}}_\mathbf{1} \mathbf{a_2}(2c)^{2t_1+k_1+2} t_1! (k_1+t_1+1)! \frac{\pi^{m/2+1}}{\Gamma(m/2)} \qquad & j_1=2t_1+1.
\end{cases}
\end{align*}
\end{theorem}
\begin{proof}
The inner product $\left\langle \psi_{j_1,k_1} \middle\vert \psi_{j_2,k_2} \right\rangle$ vanishes when $j_1$ and $j_2$ have a different parity. Indeed, repeating the procedure of \eqref{proces} leads to the operator $\tilde{D_c}^\dagger$ acting on a function $\psi_{0,k}$, which makes the inner product zero.\\
When $j_1$ and $j_2$ have the same parity, odd and even cases have to be treated separately. We first consider the case $j_1=2t_1, j_2=2t_2$ and, without loss of generality, $t_1\geq t_2$. Using equation \eqref{proces}, one gets
\begin{align*}
\langle \psi_{2t_1,k_1}, \psi_{2t_2,k_2} \rangle &= \left\langle \psi_{2t_1-1,k_1}, \tilde{D_c}^\dagger \psi_{2t_2,k_2} \right\rangle\\
&= 2ct_2 \ \left\langle \psi_{2t_1-2,k_1}, \tilde{D_c}^\dagger \psi_{2t_2-1,k_2} \right\rangle\\
&= (2c)^2t_2(k_2+t_2)\ \left\langle \psi_{2t_1-2,k_1}, \psi_{2t_2-2,k_2} \right\rangle\\
&= (2c)^{2t_2} t_2! \frac{(k_2+t_2)!}{k_2!}\ \left\langle \psi_{2(t_1-t_2),k_1}, \psi_{0,k_2} \right\rangle
\intertext{which vanishes if $t_1\neq t_2$. Using Lemma \ref{innermonogeen}, the inner product of the Clifford-Hermite functions of even degree is given by}
\langle \psi_{2t_1,k_1}, \psi_{2t_2,k_2} \rangle &= 2\overline{\mathbf{a}}_\mathbf{1} \mathbf{a_2}\ (2c)^{2t_1+k_1+1} t_1! (k_1+t_1)! \frac{\pi^{m/2+1}}{\Gamma(m/2)}\ \delta_{t_1t_2}\delta_{k_1k_2}.
\end{align*}
Now suppose $j_1=2t_1+1, j_2=2t_2+1$. Without loss of generality, we take again $t_1\geq t_2$. Using equation \eqref{proces} and the previous result, one gets
\begin{align*}
\langle \psi_{2t_1+1,k_1}, \psi_{2t_2+1,k_2} \rangle &= \left\langle \psi_{2t_1,k_1}, \tilde{D_c}^\dagger \psi_{2t_2+1,k_2} \right\rangle\\
&= 2c(k_2+t_2+1) \ \left\langle \psi_{2t_1,k_1}, \psi_{2t_2,k_2} \right\rangle\\
&= (2c)^{2t_2+1} t_2! \frac{(k_2+t_2+1)!}{k_2!}\ \left\langle \psi_{2(t_1-t_2),k_1}, \psi_{0,k_2} \right\rangle
\intertext{which vanishes if $t_1\neq t_2$. The inner product of Clifford-Hermite functions of odd degree is thus given by}
\langle \psi_{2t_1+1,k_1}, \psi_{2t_2+1,k_2} \rangle &= 2\overline{\mathbf{a}}_\mathbf{1} \mathbf{a_2}\ (2c)^{2t_1+k_1+2} t_1! (k_1+t_1+1)! \frac{\pi^{m/2+1}}{\Gamma(m/2)}\ \delta_{t_1t_2}\delta_{k_1k_2}.
\end{align*}
This proves the theorem.
\end{proof}
Analogous to their classical counterparts, the Clifford-Hermite functions $\psi_{j,k}$ satisfy a scalar differential equation. We end this section with the following theorem.
\begin{theorem}
The Clifford-Hermite functions $\psi_{j,k}$ are solutions of the scalar differential equation
\begin{align}
\left(cD_0^2 + \frac{\vert\bfx\vert^2}{4c}\right) \psi_{j,k}(\bfx) =(j+k+1)\psi_{j,k}(\bfx). \label{diffvgl2}
\end{align}
\end{theorem}
\begin{proof}
Substituting $\alpha=4c$ in equation \eqref{alpha}, we obtain
\begin{align*}
e^{-\vert\bfx\vert^2/4c} D_0 H_j(m_k)(\bfx) &=\left(D_0 + \frac{\bfx}{2c} \right) \psi_{j,k}(\bfx)
\intertext{and}
e^{-\vert\bfx\vert^2/4c} D_0^2 H_j(m_k)(\bfx) &=\left(D_0^2 + \frac{1}{2c}(D_0\bfx+\bfx D_0) +\frac{\bfx^2}{4c^2} \right) \psi_{j,k}(\bfx).
\end{align*}
Multiplying the differential equation for the Clifford-Hermite polynomials in Theorem \ref{differential} on the left with $\exp{\left(\frac{-\vert\bfx\vert^2}{4c}\right)}$ and substituting the above results, one gets
\begin{align*}
\left(cD_0^2 -(\mE+1) +\frac{\bfx^2}{4c} -\bfx D_0 - \frac{\bfx^2}{2c} +C(j,k) \right) \psi_{j,k}(\bfx) &=0
\end{align*}
with
\begin{align*}
C(j,k) =
\begin{cases}
-j \phantom{(2+k+1)} \qquad j \textrm{ even}, \\
-(2k+j+1) \qquad j \textrm{ odd}
\end{cases}
\end{align*}
as in Theorem \ref{differential}. Because of Lemma \ref{Drelaties}, we find
\begin{align*}
\mE \psi_{j,k}(\bfx)&= \sum_{i=0}^\infty \frac{1}{i!} \left(\frac{1}{4c} \right)^i (-2i \bfx^{2i-1} + \bfx^{2i} D_0) H_j(m_k)(\bfx)\\
&= \left(-\bfx D_0 +B(j,k) \right) \psi_{j,k}(\bfx)
\end{align*}
with
\begin{align*}
B(j,k) =
\begin{cases}
k\phantom{-({}+1){}} \qquad j \textrm{ even}, \\
-(k+1) \qquad j \textrm{ odd}.
\end{cases}
\end{align*}
as in Lemma \ref{evenodd}. Using this result and the $\mathfrak{osp}(1\vert2)$-relations, we have
\begin{align*}
\left(cD_0^2 -\frac{\bfx^2}{4c} -B(j,k) -1 +C(j,k) \right) \psi_{j,k}(\bfx) &=0
\end{align*}
which simplifies to
\begin{align*}
\left(cD_0^2  -\frac{\bfx^2}{4c} \right) \psi_{j,k}(\bfx) &=\left(j+k+1 \right) \psi_{j,k}(\bfx).
\end{align*}
An alternative proof can be obtained using induction.
\end{proof}
%
%
\section{Conclusions}
\setcounter{equation}{0}
\label{conc}
The slice Dirac operator as defined in this paper, revealed the $\mathfrak{osp}(1\vert2)$ Lie superalgebra structure of the theory of slice monogenic functions. Under the condition that this slice Dirac operator $D_0$ is self-adjoint, an inner product is defined on $L_{2}(\mR^{m+1},r^{1-m}\mathrm{d}\bfx)\otimes Cl_{m+1}$. Based on the kernel of $D_0$ and using the $\mathfrak{osp}(1\vert2)$-relations, it has been possible to define Clifford-Hermite polynomials $h_{j,k}$ and Clifford-Hermite functions $\psi_{j,k}$ which exhibit analogous properties to their classical counterparts. These Clifford-Hermite functions $\psi_{j,k}$ are orthogonal with respect to the inner product defined in Section \ref{hilbert} and are solutions of a scalar differential equation.
Further research will focus on the construction of a Fourier transform in this context and investigate the advantages of the $\mathfrak{osp}(1\vert2)$-structure for the study of slice monogenic functions. We will also study the definition of an inner product as a two-dimensional integral over one slice.
\section{Acknowledgements}
L. Cnudde and H. De Bie are supported by the UGent BOF starting grant 01N01513. G. Ren is partially supported by the NNSF of China (11071230, 11371337) and RFDP (20123402110068).
%



\end{document}